%% file: m.tex
\documentclass[reqno,12pt]{amsart}

\include{environ}
\include{symbols}
\def\bold#1{{\mathbf #1}}
\def\text#1{{\mathrm #1}}
\def\calg#1{{\mathcal #1}}
\def\bbbb#1{{\mathbb #1}}

\begin{document}

\title[Determining Projections and Functionals for Navier-Stokes Equations]
      {Determining Projections and Functionals for \\
       Weak Solutions of the Navier-Stokes Equations}

\author[M. Holst]{Michael Holst}
\email{holst@ama.caltech.edu}
\address{Applied Mathematics 217-50, Caltech, Pasadena, CA 91125, USA.}

\author[E. S. Titi]{Edriss Titi}
\address{Department of Mathematics and Department of Mechanical and Aerospace 
         Engineering, University of California, Irvine, CA 92697-3875, USA.}
\email{etiti@math.uci.edu}

\thanks{The first author was supported in part by the NSF under Cooperative
        Agreement No.~CCR-9120008.
        The work of the second author was supported in part by NSF Grant
        No.~DMS-93-08774 and by the University of California-Irvine
        Graduate Council Research Fund.
        The second author would like to thank the CNLS and the IGPP at the
        Los Alamos National Laboratory for their kind hospitality while this
        work was completed.}

\date{August 7, 1996}


\input{abs}

\maketitle


{\footnotesize
\tableofcontents
}
\vspace*{-0.5cm}

\input{body}

\input{app}

\bibliographystyle{abbrv}
\bibliography{m}


\vspace*{0.5cm}

\end{document}

%% file: environ.tex
\pdfoutput=1

\usepackage{color} 
\usepackage{ifpdf}
\ifpdf
    \usepackage[pdftex]{graphicx}
    \usepackage[pdftex]{hyperref}
    \hypersetup{
        unicode=false,          
        pdftoolbar=true,        
        pdfmenubar=true,        
        pdffitwindow=false,     
        pdfstartview={FitH},    
        pdftitle={MCP Article},      
        pdfauthor={Michael Holst},   
        pdfsubject={Mathematics},    
        pdfcreator={Michael Holst},  
        pdfproducer={Michael Holst}, 
        pdfkeywords={PDE, analysis, mathematical physics}, 
        pdfnewwindow=true,      
        colorlinks=true,        
        linkcolor=red,          
        citecolor=blue,         
        filecolor=magenta,      
        urlcolor=cyan           
    }

\else
    \usepackage{graphicx}
    \usepackage{pstricks}

\fi

\usepackage{times}
\usepackage{amsfonts}

\usepackage{amsmath}
\usepackage{amsthm}
\usepackage{amssymb}
\usepackage{amsbsy}
\usepackage{amscd}

\usepackage{enumerate}
\usepackage{verbatim}
\usepackage{subfigure}

\newtheorem{theorem}{Theorem}[section]
\newtheorem{corollary}[theorem]{Corollary}
\newtheorem{lemma}[theorem]{Lemma}

\newtheorem{definition}[theorem]{Definition}

\newtheorem{remark}[theorem]{Remark}

\numberwithin{equation}{section}  

  \newcounter{mnote}
  \let\oldmarginpar\marginpar
    \renewcommand\marginpar[1]{\-\oldmarginpar[\raggedleft\footnotesize #1]%
    {\raggedright\footnotesize #1}}

\definecolor{myblue}{rgb}{0.2,0.2,0.7}
\definecolor{mygreen}{rgb}{0,0.6,0}
\definecolor{mycyan}{rgb}{0,0.6,0.6}
\definecolor{myred}{rgb}{0.9,0.2,0.2}
\definecolor{mymagenta}{rgb}{0.9,0.2,0.9}
\definecolor{mywhite}{rgb}{1.0,1.0,1.0}
\definecolor{myblack}{rgb}{0.0,0.0,0.0}

%% file: symbols.tex
\newcommand{\beq}{\begin{equation}}
\newcommand{\eeq}{\end{equation}}
\newcommand{\beqa}{\begin{eqnarray}}
\newcommand{\eeqa}{\end{eqnarray}}

%% file: abs.tex
\begin{abstract}
In this paper we prove that an operator which projects weak solutions of
the two- or three-dimensional Navier-Stokes equations onto a 
finite-dimensional space is determining if it annihilates the difference 
of two ``nearby'' weak solutions asymptotically, and if it satisfies a 
single appoximation inequality.
We then apply this result to show that the long-time behavior of weak 
solutions to the Navier-Stokes equations, in both two- and three-dimensions,
is determined by the long-time behavior of a finite set of bounded 
linear functionals.
These functionals are constructed by local surface averages of solutions
over certain simplex volume elements, and are therefore well-defined for 
weak solutions.
Moreover, these functionals define a projection operator which satisfies
the necessary approximation inequality for our theory.
We use the general theory to establish lower bounds on the simplex 
diameters in both two- and three-dimensions. Furthermore, in the three
dimensional case we 
make a connection between their diameters and the
Kolmogoroff dissipation small scale in turbulent flows.
\end{abstract}

%% file: body.tex
\section{Introduction}

Consider a viscous incompressible
fluid in $\Omega \subset \bbbb{R}^d$, where $\Omega$ is an open bounded
domain with Lipschitz continuous boundary, and where $d=2$ or $d=3$.
Given the kinematic viscosity $\nu>0$, and the vector volume force 
function $f(x,t)$ for each $x \in \Omega$ and $t \in (0,\infty)$, 
the governing Navier-Stokes equations for the fluid velocity 
vector $u=u(x,t)$ and the scalar pressure field $p=p(x,t)$ are:
\begin{equation}
   \label{eqn:nv}
\frac{\partial u}{\partial t} - \nu \Delta u
    + (u \cdot \nabla) u + \nabla p = f
  \ \ \ \text{in} \ \ \Omega \times (0,\infty),
\end{equation}
\begin{equation}
   \label{eqn:nv_bc}
\nabla \cdot u = 0
  \ \ \ \text{in} \ \ \Omega \times (0,\infty).
\end{equation}
Also provided are initial conditions $u(0) = u_0$, as well as 
appropriate boundary conditions on $\partial \Omega \times (0,\infty)$.

The notion of {\em determining modes} for the Navier-Stokes equations was
first introduced in~\cite{FoPr67} as an attempt to identify and estimate the
number of degrees of freedom in turbulent flows (cf.~\cite{CFMT85} for a
thorough discussion of the role of determining sets in turbulence theory).
This concept later led to the notion of 
{\em Inertial Manifolds}~\cite{FST88}.
An estimate of the number of determining modes was given in~\cite{FMTT83} and
later improved in~\cite{JoTi93}.
The notion of {\em determining nodes}, and other more general determining
concepts, were introduced in~\cite{FoTe83}.
In~\cite{FoTe84} the notion of determining nodes was discussed in detail, 
and estimates for their number were reported in~\cite{JoTi92}, and later 
improved in~\cite{JoTi93}.
In~\cite{FoTi91} (see also~\cite{JoTi91}) the concept of 
{\em determining volume elements} was presented, and a connection was
established between this concept and Inertial Manifolds.
A generalized and unified theory of all of the above
was recently presented in~\cite{CJT95a,CJT97}.

Bounds on the number of determining modes, nodes, and volumes are usually 
phrased in terms of a generalized {\em Grashof number}, which is defined for 
the two-dimensional Navier-Stokes equations as:
$$
Gr = \frac{\rho^2 F}{\nu^2} = \frac{F}{\lambda_1 \nu^2},
$$
where $\lambda_1$ is the smallest eigenvalue of the Stokes operator
and $\rho=\sqrt{\lambda_1}$ is the related (best) Poincar\'e constant.
Here,
$F = \limsup_{t \rightarrow \infty} ( \int_{\Omega} |f(x,t)|^2 )^{1/2}$
if $f \in L^2(\Omega)$ for almost every~$t$, or
$F = \limsup_{t \rightarrow \infty} \sqrt{\lambda_1} \| f \|_{H^{-1}(\Omega)}$
if $f \in H^{-1}(\Omega)$ for almost every~$t$.

The best known estimate for the determining set size for the two-dimensional 
Navier-Stokes equations with periodic boundary conditions and $H^2$-regular 
solutions is of order $Gr$~\cite{JoTi93}.
In obtaining their estimate, the authors relied on the fact that the domain 
had no physical boundaries to shed vorticity, which made available some 
convenient properties of $H^2$-regular solutions.
However, in the two-dimensional case with no-slip boundary conditions, to
our knowledge the best estimate on the cardinal of any determining set
(modes, nodes, or volumes) that can be obtained is of 
order $Gr^2$, even for $H^2$-regular solutions.

Due to the Sobolev Imbedding Theorem $H^2 \hookrightarrow C^0$ (which holds
in dimensions 1, 2, and 3), 
or rather due to the failure of the imbedding 
$H^1 \hookrightarrow C^0$ in dimensions 2 and 3, determining node 
analysis is necessarily restricted to $H^2$-regular solutions to make
sense of point-wise values.
However, when talking about determining modes or volume elements,
it is sufficient for functions to be $H^1$-regular,
so that these concepts also make sense for weaker solutions.
To construct a general analysis framework for the case of weak 
$H^1$ solutions, we can begin by defining notions of
{\em determining projections} and {\em determining functionals}
for weak solutions.
(The standard spaces $H$, $V$, and $V'$ are reviewed fully in~\S2.)
\begin{definition} {\em
   \label{def:determining_operator}
Let $f(t), g(t) \in V'$ be any two forcing functions satisfying
\begin{equation}
    \label{eqn:forcing}
\lim_{t \rightarrow \infty} \| f(t) - g(t) \|_{V'} = 0,
\end{equation}
and let $u, v \in V$ be corresponding weak solutions 
to~(\ref{eqn:nv})--(\ref{eqn:nv_bc}).
The projection operator $R_N : V \mapsto V_N \subset L^2(\Omega)$, 
~ $N = \dim(V_N) < \infty$,
is called a {\em determining projection} for weak solutions of the
$d$-dimensional Navier-Stokes equations if
\begin{equation}
  \label{eqn:project}
\lim_{t \rightarrow \infty} 
    \| R_N( u(t) - v(t)) \|_{L^2(\Omega)} = 0,
\end{equation}
implies that
\begin{equation}
    \label{eqn:determine}
\lim_{t \rightarrow \infty} \| u(t) - v(t) \|_H = 0.
\end{equation}
} \end{definition}

Given a basis $\{ \phi_i \}_{i=1}^N$ for the finite-dimensional space $V_N$, 
and a set of bounded linear functionals $\{ l_i \}_{i=1}^N$ from $V'$, 
we can construct a projection operator as:
\begin{equation}
  \label{eqn:projectOper}
R_N u = \sum_{i=1}^N l_i(u) \phi_i.
\end{equation}
The assumption (\ref{eqn:project}) is then implied by:
\begin{equation}
  \label{eqn:linfunc}
\lim_{t \rightarrow \infty} 
    | l_i( u(t) - v(t)) | = 0, \ \ \ \ \ i=1, \ldots, N
\end{equation}
so that we can ask equivalently whether the set $\{l_i\}_{i=1}^N$ 
forms a set of {\em determining functionals}
(see~\cite{CJT95a,CJT97}).
The analysis of whether $R_N$ or $\{l_i\}_{i=1}^N$ are determining 
can be reduced to an analysis of the approximation properties of $R_N$. 
Note that in this construction,
the basis $\{\phi_i\}_{i=1}^N$ need not span a subspace of the solution 
space $V$, so that the functions $\phi_i$ need not be divergence-free
for example.
Note that Definition~\ref{def:determining_operator}
encompasses each of the notions of determining modes, nodes, and volumes
by making particular choices for the sets of functions 
$\{\phi_i\}_{i=1}^N$ and $\{l_i\}_{i=1}^N$
(see~\cite{JoTi91,JoTi92}).

In this paper, we will employ Definition~\ref{def:determining_operator}
to extend the results of~\cite{CJT95a,CJT97} to the more 
general setting of $H^1$-regular solutions.
In particular, we will show that if a projection operator 
$R_N : V \mapsto V_N \subset L^2(\Omega)$, ~ $N = \dim(V_N) < \infty$,
satisfies an approximation inequality for $\gamma > 0$ of the form,
\begin{equation}
  \label{eqn:approximation}
\| u - R_N u \|_{L^2(\Omega)} \le C_1 N^{-\gamma} \| u \|_{H^1(\Omega)},
\end{equation}
then the operator $R_N$ is a determining projection in the sense of 
Definition~\ref{def:determining_operator}, provided $N$ is large enough.
We will also derive explicit bounds on the dimension $N$ which guarantees that
$R_N$ is determining.
While we gain generality in our approach here, we also lose something in the
balance: the bounds obtained here are generally of order $Gr^2$,
whereas the bounds in~\cite{CJT95a,CJT97} (requiring $H^2$-regularity)
are of order $Gr$.

\subsection*{Outline of the paper}

Preliminary material is presented in~\S2,
including some inequalities for bounding the nonlinear term appearing 
in weak formulations of the Navier-Stokes equations.
In~\S3, a finite element interpolant due to Scott and Zhang is 
presented, which (unlike nodal interpolation) is well-defined for
$H^1$-functions.
It is shown that the interpolant satisfies the approximation 
assumption~(\ref{eqn:approximation}) for $H^1$-functions on arbitrary 
polyhedral domains in both two and three dimensions;
most of the details are relegated to the Appendix.
In~\S4, we consider the two-dimensional Navier-Stokes equations, and 
derive bounds on the dimension $N$ of the space $V_N$, employing only
the approximation assumption~(\ref{eqn:approximation}).
As an application of this general result, we employ some standard 
assumptions about simplex triangulations of the domain
(discussed in~\S3) and derive lower bounds on the simplex diameters,
sufficient to ensure that the SZ-interpolant is a determining projection 
(equivalently, that the simplex surface integrals forming SZ-interpolant 
coefficients are a determining set of linear functionals).
We extend these results to three dimensions in~\S5, 
by requiring (following~\cite{CDT95}) 
that weak solutions satisfy an additional technical assumption
(due to the lack of appropriate global {\em a priori} estimates), which is 
related to the natural notion of mean dissipation rate of energy.

\section{Preliminary Material}

We briefly review some background material following the 
notation of~\cite{CoFo88,Lion69,Tema77,Tema83}.
Let $\Omega \subset \mathbb{R}^d$ denote an open bounded set.
The imbedding results we will need are known to hold for example if the domain 
$\Omega$ has a locally Lipschitz boundary, denoted as 
$\Omega \in \calg{C}^{0,1}$ (cf.~\cite{Adam78}).
For example, open bounded convex sets $\Omega \subset \bbbb{R}^d$ 
satisfy $\Omega \in \calg{C}^{0,1}$ (Corollary~1.2.2.3 in~\cite{Gris85}),
so that convex polyhedral domains (which we consider here)
are in $\calg{C}^{0,1}$.

Let $H^k(\Omega)$ denote the usual Sobolev spaces $W^{k,2}(\Omega)$.
Employing multi-index notation, the distributional partial derivative of 
order $|\alpha|$ is denoted $D^{\alpha}$, so that the 
(integer-order) norms and semi-norms in $H^k(\Omega)$ may be denoted
$$
\| u \|^2_{H^k(\Omega)}
    = \sum_{j=0}^k |\Omega|^{\frac{j-k}{d}} | u |_{H^j(\Omega)}^2,
\ \ \ \ \ \ 
| u |^2_{H^j(\Omega)} 
    = \sum_{| \alpha | = j} \| D^{\alpha} u \|_{L^2(\Omega)},
\ \ \ 
0 \le j \le k,
$$
where $|\Omega|$ represents the measure of $\Omega$.
Fractional order Sobolev spaces and norms may be defined for example
through Fourier transform and extension theorems, or through 
interpolation.
A fundamentally important subspace is the $k=1$ case of
$$
H^k_0(\Omega) = \text{closure~of}~ C_0^{\infty}(\Omega)
        ~\text{in}~ H^k(\Omega),
$$
in which the Poincar\'e Inequality reduces to:
If $\Omega$ is bounded, then 
\begin{equation}
   \label{eqn:poincare}
\| u \|_{L^2(\Omega)} \le \rho(\Omega) | u |_{H^1(\Omega)}, 
\ \ \ \ \ 
\forall u \in H^1_0(\Omega).
\end{equation}

The spaces above extend naturally (cf.~\cite{Tema77}) to product spaces of 
vector functions $u=(u_1,u_2,\ldots,u_d)$, which are denoted with the
same letters but in bold-face; for example,
$\bold{H}^k_0(\Omega) = \left( H^k_0(\Omega) \right)^d$.
The inner-products and norms in these product spaces are extended in the 
natural Euclidean way; the convention here will be to subscript these 
extended vector norms the same as the scalar case.

Define now the space $\calg{V}$ of divergence free 
$\bold{C}^{\infty}$ vector functions
with compact support as
$$
\calg{V} = \left\{ \phi \in \bold{C}_0^{\infty}(\Omega)
            ~|~ \nabla \cdot \phi = 0 \right\}.
$$
The following two subspaces of $\bold{L}^2(\Omega)$ and $\bold{H}^1_0(\Omega)$
are fundamental to the study of the Navier-Stokes equations.
$$
H = \text{closure~of}~ \calg{V} ~\text{in}~ \bold{L}^2(\Omega),
\ \ \ \ \ \ \ \ \ \ 
V = \text{closure~of}~ \calg{V} ~\text{in}~ \bold{H}^1_0(\Omega).
$$
To simplify the notation, it is common (cf.~\cite{CoFo88,Tema77}) to use the 
following notation for inner-products and norms in $H$ and $V$:
\begin{equation}
   \label{eqn:notation}
(u,v) = (u,v)_H,
\ \ \ \ 
|u| = \|u\|_H,
\ \ \ \ 
((u,v)) = (u,v)_V,
\ \ \ \ 
\|u\| = \|u\|_V.
\end{equation}

The Navier-stokes equations~(\ref{eqn:nv})--(\ref{eqn:nv_bc}) are
equivalent to the functional differential equation:
\begin{equation}
\frac{du}{dt} + \nu Au + B(u,u) = f, \ \ \ \ \ u(0) = u_0.
   \label{eqn:strong}
\end{equation}
The Stokes operator $A$ and bilinear form $B$ are defined as
$$
Au = -P \Delta u,
\ \ \ \ \ B(u,v) = P [(u \cdot \nabla) v],
$$
where the operator $P$ is the Leray orthogonal projector,
$P :H_0^1 \mapsto V$ and $P : L^2 \mapsto H$, respectively.

Weak formulations, which we consider shortly, will use the 
bilinear Dirichlet form $((\cdot,\cdot))$ and trilinear form 
$b(\cdot,\cdot,\cdot)$ as:
$$
((u,v)) = (\nabla u, \nabla v),
\ \ \ \ \ b(u,v,w) = (B(u,v),w) = (P ((u \cdot \nabla) v), w).
$$
(Note that thanks to the Poincar\'e inequality~(\ref{eqn:poincare}), the 
form $((\cdot,\cdot))$ is
actually an inner-product on V, and the induced norm 
$\|\cdot\|=((\cdot,\cdot))^{1/2}$ is in fact a norm on V, 
equivalent to the $H^1$-norm.)

{\em A priori} bounds can be derived for the form
$b(\cdot,\cdot,\cdot)$ (cf.~\cite{CoFo88,Lady69,Tema77}).
In particular, if $\Omega \subset \bbbb{R}^d$, then
the trilinear form $b(u,v,w)$ is bounded on $V \times V \times V$ as
follows:
\begin{eqnarray}
d=2: & 
|b(u,v,w)| \le 2^{1/2} \|u\|_{L^2(\Omega)}^{1/2} |u|_{H^1(\Omega)}^{1/2} 
                      |v|_{H^1(\Omega)} 
                     \|w\|_{L^2(\Omega)}^{1/2} |w|_{H^1(\Omega)}^{1/2},
   \label{eqn:lady_2d} \\
d=3: &
|b(u,v,w)| \le 2 \|u\|_{L^2(\Omega)}^{1/4} |u|_{H^1(\Omega)}^{3/4} 
                      |v|_{H^1(\Omega)} 
                     \|w\|_{L^2(\Omega)}^{1/4} |w|_{H^1(\Omega)}^{3/4}.
   \label{lemma:lady_3d}
\end{eqnarray}
Moreover, from H\"older inequalities we have for $d=2$ or $d=3$:
\begin{equation}
   \label{eqn:l_infty}
|b(v,u,v)| \le \|\nabla u\|_{L^{\infty}(\Omega)} \|v\|_{L^2(\Omega)}^2.
\end{equation}

\section{Polynomial interpolation in $\bold{H}^1_0(\Omega)$}

An example of a projection operator which satisfies the approximation
assumption~(\ref{eqn:approximation}) is that used for defining determining 
volumes~\cite{JoTi91}; we examine now powerful alternative operator.
Let $\Omega \subset \bbbb{R}^d$ be a d-dimensional polygon, 
exactly triangulated by (for example) Delaunay triangulation~\cite{Edel87}, 
with quasi-uniform, shape-regular simplices, 
the vertices of which will form a set of $N$ generalized
interpolation points in our analysis.
Note that for quasi-uniform, shape-regular triangulations in
$\bbbb{R}^d$ (see~\cite{Ciar78} for detailed discussions), 
it holds that
\begin{equation}
   \label{eqn:h_and_N}
C_0 |\Omega| h^{-d} \le N \le C_0' |\Omega| h^{-d},
\end{equation}
where $h$ is the maximum of the diameters of the simplices,
and where $C_0$ and $C_0'$ are universal constants, 
independent of both $N$ and $h$.
The parameter $h$ will be referred to as the characteristic parameter, or 
characteristic length scale, of such a quasi-uniform shape-regular mesh.

It should be noted that given some initial triangulation 
satisfying~(\ref{eqn:h_and_N}), repeated bi-section~\cite{Bans91}
or octa-section~\cite{Zhan88} (quadra-section in 2D) of each simplex can be
performed in such a way as to guarantee non-degeneracy asymptotically, in 
that the quasi-uniformity and shape-regularity are preserved.
Therefore, inequality~(\ref{eqn:h_and_N}) can be made to hold, for the
same universal constants, for finer and finer meshes in a nested sequence 
of simplex triangulations.

To properly define a continuous piecewise-linear nodal 
interpolant of a function $u \in H^1(\Omega)$ based on the nodes of a 
triangulation of $\Omega$, the particular function $u$ must 
be bounded point-wise.
This will be true if the function $u$ is continuous in $\Omega$, 
hence uniformly continuous on $\bar{\Omega}$.
One of the Sobolev imbedding results (cf.~\cite{Adam78}) 
states that if $\Omega \subset \bbbb{R}^d$ satisfies 
$\Omega \in \calg{C}^{0,1}$, then for nonnegative real numbers $k$ and $s$ 
it holds that $H^k(\Omega) \hookrightarrow C^s(\bar{\Omega})$,
$k > s + \frac{d}{2}$.
This implies that for $d=1$, the interpolant can be correctly defined,
since $H^1(\Omega)$ is continuously imbedded in $C^0(\bar{\Omega})$.
However, in higher dimensions, 
$H^{1+\alpha}(\Omega) \hookrightarrow C^0(\bar{\Omega})$
only if $\alpha>0$ when $d=2$, or if $\alpha>1/2$ when $d=3$.
While it may be possible to use the nodal interpolant and a regularity
assumption such as $u \in H^{1+\alpha}(\Omega)$ for appropriate $\alpha>0$,
an alternative approach is taken here.

The generalized interpolant due to Scott and Zhang~\cite{ScZh90} 
can be defined for $H^1$-functions in both two and three spatial dimensions.
The SZ-interpolant $I_h$ is constructed from a combination linear interpolation
and local averaging on faces and edges of simplices, and has optimal
approximation properties even in the case of $H^1$-functions.
\begin{lemma}
   \label{lemma:sz_vector}
For the SZ-interpolant of $u\in \bold{H}_0^{1+\alpha}(\Omega)$, 
$\alpha \ge 0$, it holds that
$$
\|u-I_hu\|_{L^2(\Omega)} \le C_1 h^{1+\alpha} | u |_{H^{1+\alpha}(\Omega)}.
$$
\end{lemma}
\begin{proof}
See the appendix for a condensed proof following~\cite{BrSc94,ScZh90}.
\end{proof}

Note that both the usual nodal interpolant and the 
SZ-interpolant $I_h$ can be written as a linear combination of linear 
functionals:
$$
I_h u(x) = \sum_{i=1}^N \phi_i(x) l_i(u).
$$
In either case, the set of functions $\{\phi_i\}_{i=1}^N$ is the usual 
continuous piecewise-polynomial nodal finite element basis defined over the 
simplicial mesh, satisfying the {\em Lagrange} property at the vertices of the 
mesh:
$$
\phi_i(x_j) = \delta_{ij}.
$$
The difference between the two interpolants is simply the choice of the
linear functionals: in the case of the nodal interpolant, the 
functionals are delta functions centered at the vertices of the mesh; 
in the case of the SZ-interpolant, they are defined in terms of a 
bi-orthogonal dual basis (see the Appendix).

\section{The Two-dimensional Navier-Stokes Equations}

A general weak formulation of the Navier-Stokes 
equations~(\ref{eqn:nv})--(\ref{eqn:nv_bc}) can be
written as~(cf.~\cite{CoFo88,Tema77}):
\begin{definition} {\em
Given $f \in L^2([0,T];V')$, a weak solution of the Navier-Stokes equations
satisfies $u \in L^2([0,T];V) \cap C_w([0,T];H)$, 
$du/dt \in L^1_{\text{loc}}((0,T];V')$, and
\begin{equation}
   \label{eqn:weak}
< \frac{du}{dt} ,v >
    + \nu ((u,v)) + b(u,u,v) = <f,v>, 
      \ \ \forall v \in V, 
      \ \ \text{~for~almost~every~} t,
\end{equation}
\begin{equation}
   \label{eqn:weak_ic}
u(0) = u_0.
\end{equation}
} \end{definition}
Here, the space $C_w([0,T];H)$ is the subspace of $L^{\infty}([0,T];H)$ of
weakly continuous functions, and $<\cdot,\cdot>$ denotes the duality pairing
between $V$ and $V'$, where $H$ is the Riesz-identified pivot space in the
Gelfand triple $V \subset H=H' \subset V'$.
Note that since the Stokes operator can be uniquely extended 
to $A : V \mapsto V'$, and since it can be shown that
$B : V \times V \mapsto V'$ (cf.~\cite{CoFo88,Tema83} for both results),
the functional form~(\ref{eqn:strong}) still makes sense for weak 
solutions, and the total operator represents a mapping $V \mapsto V'$.

In the two-dimensional case, for a forcing function 
$f \in L^{\infty}([0,T];V')$, there exists a unique weak solution 
$u \in L^2([0,T];V) \cap C_w([0,T];H)$ (cf.~\cite{CoFo88,Tema83}).
Consider now two forcing functions $f, g \in L^2([0,\infty];V')$ and
corresponding weak solutions $u$ and $v$ to~(\ref{eqn:strong}) 
in either the two- or three-dimensional case.
Subtracting the equations~(\ref{eqn:strong}) for $u$ and $v$ yields
an equation for the difference function $w=u-v$, namely
\begin{equation}
  \label{eqn:above}
\frac{dw}{dt} + \nu A w + B(u,u) - B(v,v) = f-g.
\end{equation}
Since the residual of equation~(\ref{eqn:above}) lies in the dual space $V'$,
for almost every $t$, 
we can consider the dual pairing of each side~(\ref{eqn:above}) with a function 
in $V$, and in particular with $w \in V$, which yields
$$
<\frac{dw}{dt}, w>
    + \nu \|w\|^2 + b(u,u,w) - b(v,v,w) = <f-g,w> \ \ \text{~for~almost~every~} t.
$$
It can be shown (cf.~\cite{Tema77}, Chapter~3, Lemma~1.2) that
$$
\frac{1}{2} \frac{d}{dt} |w|^2 = <\frac{dw}{dt}, w>
$$
in the distribution sense.
It can also be shown~\cite{CoFo88,Tema77} that $b(u,v,w) = -b(u,w,v)$, 
$\forall u,v,w \in V$, so that $b(w,u,w) = b(u,u,w) - b(v,v,w)$.
Therefore, the function $w=u-v$ must satisfy
\begin{equation}
   \label{eqn:weak_diff}
\frac{1}{2}\frac{d}{dt} |w|^2
   + \nu \|w\|^2 + b(w,u,w) = <f-g,w>.
\end{equation}

The following generalized Gronwall inequality will be a key tool in the 
analysis to follow (see~\cite{FMTT83} and~\cite{JoTi91}).
\begin{lemma}
   \label{lemma:gronwall_2}
Let $T>0$ be fixed, and let $\alpha(t)$ and $\beta(t)$ be locally integrable 
and real-valued on $(0,\infty)$, satisfying:
$$
\liminf_{t \rightarrow \infty}
    \frac{1}{T} \int_{t}^{t+T} \alpha(\tau) d\tau = m>0,
\ \ \ \ \ 
\limsup_{t \rightarrow \infty}
    \frac{1}{T} \int_{t}^{t+T} \alpha^-(\tau) d\tau = M<\infty,
$$
$$
\lim_{t \rightarrow \infty}
     \frac{1}{T} \int_{t}^{t+T} \beta^+(\tau) d\tau = 0,
$$
where $\alpha^-=\max\{-\alpha,0\}$ and $\beta^+=\max\{\beta,0\}$.
If $y(t)$ is an absolutely continuous non-negative function on
$(0,\infty)$, and $y(t)$ satisfies the following differential inequality:
$$
y'(t) + \alpha(t) y(t) \le \beta(t), \ \ \text{a.e.~on}~ (0,\infty),
$$
then $\lim_{t \rightarrow \infty} y(t) = 0$.
\end{lemma}

The main two-dimensional results are now given;
we assume that $\Omega \subset \bbbb{R}^2$ is 
an open bounded domain with Lipschitz continuous boundary.
\begin{theorem}
   \label{theo:main1}
Let $f(t), g(t) \in V'$ be any two forcing functions satisfying
$$
\lim_{t \rightarrow \infty} \| f(t) - g(t) \|_{V'} = 0,
$$
and let $u, v \in V$ be the corresponding weak solutions 
to~(\ref{eqn:nv})--(\ref{eqn:nv_bc}) for $d=2$.
If there exists a projection operator 
$R_N : V \mapsto V_N \subset L^2(\Omega)$, $N = \dim(V_N)$,
satisfying
$$
\lim_{t \rightarrow \infty} 
    \| R_N( u(t) - v(t)) \|_{L^2(\Omega)} = 0,
$$
and satisfying for $\gamma>0$ the approximation inequality
$$
\| u - R_N u \|_{L^2(\Omega)} \le C_1 N^{-\gamma} \| u \|_{H^1(\Omega)},
$$
then
$$
\lim_{t \rightarrow \infty} | u(t) - v(t) | = 0
$$
holds if $N$ is such that
$$
\infty > N > C \left( \frac{1}{\nu^2} 
     \limsup_{t \rightarrow \infty} \|f(t)\|_{V'}
     \right)^{\frac{1}{\gamma}},
$$
where $C$ is a constant independent of $\nu$ and $f$.
\end{theorem}
\begin{proof}
Using the notation~(\ref{eqn:notation}), we begin with 
equation~(\ref{eqn:weak_diff}), employing the inequality~(\ref{eqn:lady_2d})
along with Cauchy-Schwarz and Young's inequalities to yield 
$$
\frac{1}{2} \frac{d}{dt} |w|^2 + \nu \| w \|^2
    \le \|u\|
     ~|w|~ \| w \|
    + \|f-g\|_{V'} \|w\|
$$
$$
\le \frac{1}{\nu} \|u\|^2 |w|^2
  + \frac{1}{\nu} \|f-g\|_{V'}^2
  + \frac{\nu}{2} \|w\|^2.
$$
Equivalently, this is
$$
\frac{d}{dt} |w|^2 + \nu \| w \|^2
 -  \frac{2}{\nu} \|u\|^2 |w|^2
\le \frac{2}{\nu} \|f-g\|_{V'}^2.
$$
To bound the second term on the left from below, we employ 
the approximation assumption on $R_N$, or rather the following inequality
which follows from it:
$$
|w|^2 \le 2 N^{-2\gamma} C_1^2 \|w\|^2 + 2 \|R_N w\|_{L^2(\Omega)}^2,
$$
which yields
$$
\frac{d}{dt} |w|^2 
  +  \left( \frac{\nu N^{2\gamma}}{2 C_1^2} 
  -   \frac{2}{\nu} \|u\|^2 \right)
      |w|^2 
\le \frac{2}{\nu} \|f-g\|_{V'}^2
+ \frac{\nu N^{2\gamma}}{C_1^2} \|R_N w\|_{L^2(\Omega)}^2.
$$
This is of the form
$$
\frac{d}{dt} |w|^2 + \alpha |w|^2 \le \beta,
$$
with obvious definition of $\alpha$ and $\beta$.

The generalized Gronwall Lemma~\ref{lemma:gronwall_2} can now be applied.
Recall that both $\|f-g\|_{V'} \rightarrow 0$ and
$\|R_N w\|_{L^2(\Omega)} \rightarrow 0$ 
as $t \rightarrow \infty$ by assumption.
Since it is assumed that $u$ and $v$, and hence $w$, are in $V$,
so that all other terms appearing in $\alpha$ and $\beta$ remain bounded,
it must hold that
$$
\lim_{t \rightarrow \infty}
     \frac{1}{T} \int_{t}^{t+T} \beta^+(\tau) d\tau = 0,
\ \ \ \ \ \ \ \ 
\limsup_{t \rightarrow \infty}
    \frac{1}{T} \int_{t}^{t+T} \alpha^-(\tau) d\tau <\infty.
$$
It remains to verify that for some fixed $T>0$,
$$
\limsup_{t \rightarrow \infty}
    \frac{1}{T} \int_{t}^{t+T} \alpha(\tau) d\tau >0.
$$
This means we must verify the following inequality for some fixed $T > 0$:
\begin{equation}
   \label{eqn:key}
N^{2\gamma}
> 
\frac{2 C_1^2}{\nu} \left( \limsup_{t \rightarrow \infty}
\frac{1}{T} \int_{t}^{t+T} \frac{2 \|u\|^2}{\nu} d\tau
\right)
= \frac{4 C_1^2}{\nu^2} 
\limsup_{t \rightarrow \infty} \frac{1}{T}
      \int_{t}^{t+T} \|u\|^2 d\tau.
\end{equation}
The following {\em a priori} bound on any weak solution 
can be shown to hold (this is a simple generalization to
$f \in V'$ of the bound 
in~\cite{CoFo88} for $f \in H$):
$$
\limsup_{t \rightarrow \infty} \frac{1}{T} \int_{t}^{t+T} 
   \|u(\tau)\|^2 d\tau \le \frac{2}{\nu^2}
      \limsup_{t \rightarrow \infty} \|f(t)\|_{V'}^2,
$$
for $T = \rho^2/\nu > 0$,
where $\rho$ is the best constant from the Poincar\'e 
inequality~(\ref{eqn:poincare}).
Therefore, if
\begin{equation}
N^{2\gamma}
> 8 C_1^2 \left( 
  \frac{1}{\nu^2} 
   \limsup_{t \rightarrow \infty} \|f(t)\|_{V'}
\right)^2
\ge \frac{4 C_1^2}{\nu^2}
   \left( \frac{2}{\nu^2}
  \limsup_{t \rightarrow \infty} \|f(t)\|_{V'}^2
\right),
\end{equation}
implying that~(\ref{eqn:key}) holds, then by the Gronwall 
Lemma~\ref{lemma:gronwall_2}, it follows that
$$
\lim_{t \rightarrow \infty} |w(t)|
   = \lim_{t \rightarrow \infty} | u(t) - v(t) | = 0.
$$
\end{proof}

Assume now that $\Omega \subset \bbbb{R}^2$ is also polyhedral, 
and can be exactly triangulated with a quasi-uniform, 
shape-regular set of simplices of maximal diameter 
$h = O (N^{-1/2})$, where $N$ is the number of vertices in the
triangulation (see~\S3).
As an application of the general result above, 
we establish a lower bound on the simplex diameters of such a triangulation,
which ensures that the SZ-interpolant is a determining projection 
(equivalently, that the simplex surface integrals forming SZ-interpolant 
coefficients are a determining set of linear functionals).
\begin{corollary}
The SZ-interpolant is determining for the two-dimensional
Navier-Stokes equations if the diameter $h$ of the simplices is 
small enough so that
$$
\infty > h^{-2} > C \left( \frac{1}{\nu^2} 
     \limsup_{t \rightarrow \infty} \|f(t)\|_{V'}
     \right)^2.
$$
\end{corollary}
\begin{proof}
Since $h = O(N^{-1/2})$ for quasi-uniform, shape-regular 
triangulations in two dimensions,
taking $\alpha=0$ in Lemma~\ref{lemma:sz_vector} yields
$$
\|u-I_hu\|_{L^2(\Omega)} \le C_1 h | u |_{H^1(\Omega)}
                         \le \tilde{C}_1 N^{-1/2} \| u \|_{H^1(\Omega)}.
$$
Therefore, the SZ-interpolant $I_h$ satisfies the approximation 
inequality~(\ref{eqn:approximation}) for $\gamma = 1/2$.
The corollary then follows by application of Theorem~\ref{theo:main1}.
\end{proof}
\begin{remark}
If $f \in H$, then we have in fact a strong solution, 
i.e. $u \in H^2(\Omega)$, and the interpolation Lemma~\ref{lemma:sz_vector} 
may be applied with $\alpha=1$.
This falls into the theoretical framework of~\cite{CJT95a,CJT97}, and in
the periodic case they have shown that $N \approx Gr$, whereas the above
result for the no-slip case states that $N \approx Gr^2$.
Whether the no-slip case may be improved to $N \approx Gr$ with additional
regularity ($f \in H$) is unclear, due to the lack of an analogous
identity to
$$
(B(w,w), Aw) = 0,
$$
which holds for the two-dimensional periodic case.
In physical terms, in two dimensions this identity illustrates the lack of a 
boundary vorticity shedding source when the boundary is absent.
\end{remark}

\section{The Three-dimensional Navier-Stokes Equations}

The lack of appropriate {\em a priori} estimates in the three-dimensional case
requires a modification of the approach taken for the two-dimensional case
in the previous section.
However, the interpolation results we have employed are dimension-independent, 
and by following the analysis approach of~\cite{CDT95} very closely, we can 
obtain similar results for the three-dimensional case.
Again we require only that $f \in V'$, but we also assume the existence of a 
unique weak solution to the three-dimensional Navier-Stokes equations.
An additional technical assumption is that some measure of the mean rate
of energy dissipation be finite, namely:
$$
\epsilon_{\infty} = \inf_{T>0} \limsup_{t \rightarrow \infty}
   \frac{\nu}{T} \int_t^{t+T} \| \nabla u \|_{\infty} d\tau < \infty.
$$
This assumption implies that eventually the weak solution for the 
three-dimensional Navier-Stokes equations becomes unique, and also in the
case $f \in H$ the weak solution eventually becomes strong.
But this assumption does not imply anything about the transients, since the
quantity is required to be finite only for large time.
We assume again that $\Omega \subset \bbbb{R}^3$ is 
an open bounded domain with Lipschitz continuous boundary.
\begin{theorem}
   \label{theo:main2}
Let $f(t), g(t) \in V'$ be any two forcing functions satisfying
$$
\lim_{t \rightarrow \infty} \| f(t) - g(t) \|_{V'} = 0,
$$
and let $u, v \in V$ be the corresponding weak solutions 
to~(\ref{eqn:nv})--(\ref{eqn:nv_bc}) for $d=3$.
If there exists a projection operator 
$R_N : V \mapsto V_N \subset L^2(\Omega)$, $N = \dim(V_N)$,
satisfying
$$
\lim_{t \rightarrow \infty} 
    \| R_N( u(t) - v(t)) \|_{L^2(\Omega)} = 0,
$$
and satisfying for $\gamma>0$ the approximation inequality
$$
\| u - R_N u \|_{L^2(\Omega)} \le C_1 N^{-\gamma} \| u \|_{H^1(\Omega)},
$$
then
$$
\lim_{t \rightarrow \infty} | u(t) - v(t) | = 0
$$
holds if $N$ is such that
$$
\infty > N > C \left( \frac{1}{\nu} \inf_{T>0} \left\{
     \limsup_{t \rightarrow \infty} \frac{1}{T}
         \int_t^{t+T} \| \nabla u(s) \|_{L^\infty(\Omega)} ds
     \right\} \right)^{\frac{1}{2\gamma}},
$$
where $C$ is a constant independent of $\nu$, $f$, and $u$.
\end{theorem}
\begin{proof}
Beginning with equation~(\ref{eqn:weak_diff}), the 
inequality~(\ref{eqn:l_infty}) is employed along with Cauchy-Schwarz 
and Young's inequalities to yield
$$
\frac{1}{2} \frac{d}{dt}|w|^2 + \nu \| w \|^2
    \le \| \nabla u\|_{L^\infty(\Omega)} 
     |w|^2
    + \|f-g\|_{V'} \|w\|
$$
$$
\le  \|\nabla u\|_{L^{\infty}} |w|^2
   + \frac{1}{2 \nu} \|f - g \|_{V'}^2
   + \frac{\nu}{2} \|w\|^2
$$
Equivalently,
$$
\frac{d}{dt}|w|^2 
  + \nu \| w \|^2
  - \| \nabla u\|_{L^\infty(\Omega)} |w|^2
\le \frac{1}{\nu} \|f-g\|_{V'}^2.
$$
To bound the second term on the left from below, we employ 
a consequence of the approximation assumption on $R_N$, namely the inequality
$$
|w|^2 \le 2 N^{-2\gamma} C_1^2 \|w\|^2 + 2 \|R_N w\|_{L^2(\Omega)}^2,
$$
which yields
$$
\frac{d}{dt} |w|^2 
  +  \left( \frac{\nu N^{2\gamma}}{2 C_1^2} 
  -   \|\nabla u\|_{L^{\infty}} \right)
      |w|^2 
\le \frac{1}{\nu} \|f-g\|_{V'}^2
+ \frac{\nu N^{2\gamma}}{C_1^2} \|R_N w\|_{L^2(\Omega)}^2.
$$
This has the form
$$
\frac{d}{dt} |w|^2 + \alpha |w|^2 \le \beta,
$$
with again obvious definition of $\alpha$ and $\beta$.

The analysis now proceeds exactly as in the proof of Theorem~\ref{theo:main1},
so that all that remains is to check again that for some fixed $T > 0$,
$$
\limsup_{t \rightarrow \infty}
    \frac{1}{T} \int_{t}^{t+T} \alpha(\tau) d\tau >0.
$$
Thus, we must prove our assumption on $N$ guarantees for
a fixed $T>0$ that
\begin{equation}
   \label{eqn:key2}
N^{2\gamma}
> \frac{2 C_1^2}{\nu} 
\limsup_{t \rightarrow \infty} \frac{1}{T}
      \int_{t}^{t+T} \| \nabla u\|_{L^\infty(\Omega)} d\tau.
\end{equation}
If we select $T_* > 0$ such that
$$
2 \inf_{T>0} \left( 
      \limsup_{t \rightarrow \infty} \frac{1}{T}
      \int_t^{t+T} \| \nabla u(s) \|_{L^\infty(\Omega)} ds
     \right)
\ge
      \limsup_{t \rightarrow \infty} \frac{1}{T_*}
      \int_t^{t+T_*} \| \nabla u(s) \|_{L^\infty(\Omega)} ds,
$$
then our assumption gives
\begin{equation}
N^{2\gamma} 
> 
\frac{4 C_1^2}{\nu}
\inf_{T_*>0} \left( 
      \limsup_{t \rightarrow \infty} \frac{1}{T_*}
      \int_t^{t+T_*} \| \nabla u(s) \|_{L^\infty(\Omega)} ds
     \right)
\end{equation}
which implies~(\ref{eqn:key2}).
The theorem then follows by the Gronwall Lemma~\ref{lemma:gronwall_2}.
\end{proof}

Assume now that $\Omega \subset \bbbb{R}^3$ is also polyhedral, 
and can be exactly triangulated with a quasi-uniform, 
shape-regular set of simplices of maximal diameter 
$h = O (N^{-1/3})$, where $N$ is the number of vertices in the
triangulation.
As an application of the general three-dimensional result above, 
we will establish a lower bound on the simplex diameters of such a 
triangulation, which ensures that the SZ-interpolant is a determining 
projection
(and that the simplex surface integrals forming SZ-interpolant 
coefficients are a determining set of linear functionals).
\begin{corollary}
The SZ-interpolant is determining for the three-dimensional
Navier-Stokes equations if the diameter $h$ of the simplices is 
small enough so that
$$
\infty > h^{-2} > C \left( \frac{1}{\nu} \inf_{T>0} \left\{
     \limsup_{t \rightarrow \infty} \frac{1}{T}
         \int_t^{t+T} \| \nabla u(s) \|_{L^\infty(\Omega)} ds
     \right\} \right).
$$
\end{corollary}
\begin{proof}
Since $h = O(N^{-1/3})$ for quasi-uniform, shape-regular 
triangulations in three dimensions,
taking $\alpha=0$ in Lemma~\ref{lemma:sz_vector} yields
$$
\|u-I_hu\|_{L^2(\Omega)} \le C_1 h | u |_{H^1(\Omega)}
                         \le \tilde{C}_1 N^{-1/3} \| u \|_{H^1(\Omega)}.
$$
Therefore, the SZ-interpolant $I_h$ satisfies the approximation 
inequality~(\ref{eqn:approximation}) for $\gamma = 1/3$.
The corollary then follows by application of Theorem~\ref{theo:main2}.
\end{proof}

%% file: app.tex
\section*{Appendix: Approximability of the Scott-Zhang Interpolant}

We will sketch the proof of the approximability result for the 
SZ-interpolant given as Lemma~\ref{lemma:sz_vector}; we will follow
quite closely the proof given in~\cite{BrSc94,ScZh90}.
As throughout this paper, we assume that $\Omega \in \calg{C}^{0,1}$, and
that the given exact simplicial triangulation of $\Omega$ is both 
shape-regular and quasi-uniform.

The proof of Lemma~\ref{lemma:sz_vector} will follow easily from 
the following result (see the comments at the end of this appendix).
\begin{lemma}
   \label{lemma:sz_scalar}
For the SZ-interpolant of $u\in H_0^{1+\alpha}(\Omega)$, 
$\alpha \ge 0$, it holds that
$$
\|u-I_hu\|_{L^2(\Omega)} \le C_1 h^{1+\alpha} | u |_{H^{1+\alpha}(\Omega)}.
$$
\end{lemma}
To prove Lemma~\ref{lemma:sz_scalar}, we will begin by defining carefully the 
SZ-interpolant.
Let $\calg{T}_h = \{ \tau_i \}_{i=1}^L$ 
be the given quasi-uniform, shape-regular
mesh of $d$-simplices which exactly triangulate the underlying domain 
$\Omega$, and let $\Omega_h = \{ x_i \}_{i=1}^N$ be the set of vertices of 
these $d$-simplices.
Define
$$
V_h = \text{span} \{ \phi_i(x) \}_{i=1}^N \subset H^1(\Omega),
$$
where $\{ \phi_i(x) \}$ is the set of standard continuous piecewise linear 
(nodal) basis functions.
The nodal basis satisfies the Lagrange relationship at the vertices
(which are exactly the ``nodes'' in this setting):
$$
\phi_i(x_j) = \delta_{ij}.
$$
Now, for each vertex $x_i$, we select (arbitrarily) an associated 
$(d-1)$-simplex $\sigma_i$ from the given simplicial mesh satisfying only:
\begin{center}
(1) $x_i \in \bar{\sigma}_i$,
\ \ \ \ and \ \ \ \
(2) $\sigma_i \subset \partial \Omega$ if $x_i \in \partial \Omega$.
\end{center}
In other words, for a given vertex $x_i$ we pick an arbitrary $(d-1)$-simplex
from edges or faces of the $d$-simplices which contain $x_i$ as a vertex.
In two-dimensions, we are picking the edge of one of the triangles that
have $x_i$ as a vertex; in three-dimensions, we are picking the face of one
of the tetrahedra which have $x_i$ as a vertex.
The only restriction on this choice is near the boundary: if $x_i$ is on
the boundary, then the $(d-1)$-simplex we pick must be one of the edges or
faces of the a simplex which lies exactly on the boundary (such a choice is
always possible).

In each $(d-1)$-simplex $\sigma_i$, we number the generating vertex $x_i$
first in the set of vertices of $\sigma_i$, denoted $ \{ x_{i,j} \}_{j=1}^d$.
(I.e., we set $x_{i,1} = x_i$.)
For each $\sigma_i$, we also have a $(d-1)$-dimensional nodal basis
$\{ \phi_{i,j} \}_{j=1}^d$, where again we set $\phi_{i,1} = \phi_i$.
There exists an associated $L^2(\sigma_i)$-dual (bi-orthogonal) 
basis $\{\psi_{i,j}\}$ satisfying
$$
\int_{\sigma_i} \psi_{i,j}(x) \phi_{i,k}(x) dx = \delta_{jk},
\ \ \ \ j,k = 1,\ldots,d.
$$
Again we take $\psi_{i,1} = \psi_i, \ \forall x_i \in \Omega_h$.
Note that $\psi_i$ and $\phi_j$ also satisfy a bi-orthogonal relationship,
namely
$\int_{\sigma_i} \psi_i \phi_j dx = 0, \ i \ne j$.
We define now the SZ-interpolant as
$$
I_h : H^1(\Omega) \mapsto V_h(\Omega),
\ \ \ \ \ 
I_h u(x) = \sum_{i=1}^N \phi_i(x) l_i(u),
\ \ \ \ \ l_i(u) = \int_{\sigma_i} \psi_i(\xi) u(\xi) d\xi.
$$
Thanks to the Trace Theorem~\cite{Adam78}, 
the interpolant $I_h u(x)$ is well-defined
at nodal values even for $u \in H^1(\Omega)$, since 
$H^1(\Omega) \hookrightarrow L^2(\sigma_i)$.
Almost by construction, one can show~\cite{ScZh90} that
\begin{itemize}
\item $I_h : H^1(\Omega) \mapsto V_h(\Omega)$ is a projection
\item $I_h : H^1_0(\Omega) \mapsto V_{0h}(\Omega)$
\end{itemize}
where $V_{0h}$ is the subset of $V_h$ having zero trace on the boundary of
$\Omega$.
Thus, $I_h$ preserves homogeneous Dirichlet boundary conditions.
Using homogeneity arguments, the following stability result for the 
interpolant is established in~\cite{ScZh90}.
\begin{lemma}
   \label{lemma:stability}
For any $\tau \in \calg{T}_h$, if the support region of $\tau$ is defined as
the set
$S_{\tau} = \text{interior}~\left( \cup \{
          \bar{\tau}_i \ | \ \bar{\tau}_i \cap \bar{\tau} \ne \emptyset,
          \ \tau_i \in \calg{T}_h \} \right)$, then it holds that
$$
\| I_h u \|_{H^m(\tau)}
  \le C \sum_{k=0}^l h^{k-m} |u|_{H^k(S_{\tau})},
  \ \ 0 \le m \le l, \ \ l > 1/2.
$$
\end{lemma}
\begin{proof}
See the proof of Theorem~3.1 in~\cite{ScZh90}.
\end{proof}

The proof of the Scott and Zhang~\cite{ScZh90} approximation result
is as follows.
\begin{proof} (Lemma~\ref{lemma:sz_scalar})
Since $I_h$ is a projector from $H^1(\Omega)$ onto $V_h(\Omega)$, it follows
that on each element, $I_h$ is a projector from $H^1(\tau)$ onto
$\calg{P}_1(\tau)$, the space of linear polynomials over $\tau$.
Thus, $I_h p = p, \ \forall p \in \calg{P}_1(\tau)$, and employing also the
stability result in Lemma~\ref{lemma:stability} we have that
for $0 \le m \le k \le 2$,
$$
\|u - I_h u\|_{H^m(\tau)}
   \le \| u - p \|_{H^m(\tau)} + \| I_h(p - u) \|_{H^m(\tau)}
   \le C \sum_{k=0}^m h^{k-m} \| u - p \|_{H^k(S_{\tau})},
$$
where $S_{\tau}$ is the element support region surrounding $\tau$ as defined in
Lemma~\ref{lemma:stability}.
Employing the modified Bramble-Hilbert lemma developed in~\cite{DuSc80}
to estimate the terms of the sum gives
$$
\inf_{p \in \calg{P}_1(\tau)} \| u - p \|_{H^m(S_{\tau})}
   \le C h^{k-m} | u |_{H^k(S_{\tau})}, 
  \ \ 0 \le m \le k \le 2,
$$
where due to the assumptions about the domain and the mesh, the constant
$C$ depends only on the spatial dimension $d$.
Together with the equation above this is
$$
\| u -I_h u \|_{H^m(\tau)} \le C h^{k-m} | u |_{H^k(S_{\tau})}
  \ \ 0 \le m \le k \le 2.
$$
Since the set
$$
Q = \sup_{\tau \in \calg{T}_h}
\{ \text{card} \{ \tau \in \calg{T}_h | \tau \cap S_{\tau} \ne \emptyset \} \}
$$
is finite due to the quasi-uniformity and shape-regularity of the mesh, we
have finally that for $0 \le m \le k \le 2$, it holds that
$$
\| u - I_h u \|_{H^m(\Omega)}^2
= \sum_{\tau \in \calg{T}_h}
    \| u - I_h u \|_{H^m(\tau)}^2
\le C h^{2(k-m)} \| u \|_{H^k(\Omega)}^2.
$$
The result for non-integer exponents $k$ and $m$ follows by the usual
norm interpolation arguments between $L^2(\Omega)$ and $H^2(\Omega)$, 
which completes the proof.
\end{proof}

Lemma~\ref{lemma:sz_scalar} can be easily extended to the vector case, which 
provides finally the proof of Lemma~\ref{lemma:sz_vector}.
\begin{proof} (Lemma~\ref{lemma:sz_vector})
For $u \in \bold{H}_0^{1+\alpha}(\Omega) = ( H_0^{1+\alpha}(\Omega) )^d$,
we have that
$$
\| u - I_h u \|_{L^2}^2 
   = \sum_{i=1}^d \| u_i - I_h^{(i)} u_i \|_{L^2(\Omega)}
   \le C_1^2 h^{2(1+\alpha)} \sum_{i=1}^d | u_i |_{H^{1+\alpha}(\Omega)}^2,
$$
where $I_h^{(i)}$ denotes the scalar SZ-interpolant applied to $u_i$.
Thus,
$$
\| u - I_h u \|_{L^2(\Omega)} \le C_1 h^{1+\alpha} |u|_{H^{1+\alpha}(\Omega)}.
$$
\end{proof}